\documentclass{amsart}
\author{Julien Melleray}
\address{Universit\'e Claude Bernard -- Lyon 1 \\
  Institut Camille Jordan, CNRS UMR 5208 \\
  43 boulevard du 11 novembre 1918 \\
  69622 Villeurbanne Cedex \\
  France}
\usepackage{amssymb}
\usepackage[initials,shortalphabetic]{amsrefs}
\usepackage[all]{xy}
\usepackage{mathpazo}
\usepackage{hyperref}
\usepackage{bm}
\usepackage{my-macros}
\numberwithin{equation}{section}

\newcommand{\Coll}{\bm{\mathrm{Col}}}
\newcommand{\topp}{\mathrm{top}}
\newcommand{\base}{\mathrm{base}}
\usepackage{csquotes}

\title{Generic properties of homeomorphisms preserving a given dynamical simplex}

\begin{document}
\begin{abstract}
Given a dynamical simplex $K$ on a Cantor space $X$, we consider the set $G_K^*$ of all homeomorphisms of $X$ which preserve all elements of $K$ and have no nontrivial clopen invariant subset. Generalising a theorem of Yingst, we prove that for a generic element $g$ of $G_K^*$ the set of invariant measures of $g$ is equal to $K$. We also investigate when there exists a generic conjugacy class in $G_K^*$ and prove that this happens exactly when $K$ has only one element, which is the unique invariant measure associated to some odometer; and that in that case the conjugacy class of this odometer is generic in $G_K^*$.

\end{abstract}
\maketitle


\section{Introduction}

This work is motivated by a recent article of Yingst \cite{Yingst2020}, concerning the generic behavior (in the sense of Baire category) of some homeomorphisms of the Cantor space, which here and throughout this paper we denote by $X$. We extend some results of \cite{Yingst2020} to the setting of dynamical simplices; before explaining what those are, let us give some context.

Yingst considers some Bernoulli measures $\mu_r$, i.e. measures on $\{0,1\}^\N$ which are a countable product of measures $r \delta_0 +(1-r) \delta_1$ for some $r \in \, ]0,1[$. Say that such a measure is \emph{refinable} if for any two disjoint clopen subsets $A_1,A_2$ and any clopen subset $B$ such that $\mu(A_1)+\mu(A_2)=\mu(B)$, there exist disjoint clopen subsets $\tilde A_1$, $\tilde A_2$ of $B$ such that $\mu(A_1)=\mu(\tilde A_1)$ and $\mu(A_2)=\mu(\tilde A_2)$.
Refinability of a full atomless measure is a weakening of the notion of \emph{good measure}, considered by Akin \cite{Akin2005} : a Borel probability measure $\mu$ on $X$ is \emph{good} if it is full, atomless and, for any clopen $A,B$ such that $\mu(A) \le \mu(B)$ there exists a clopen $\tilde A \subseteq B$ such that $\mu(\tilde A)=\mu(A)$. Glasner and Weiss \cite{Glasner1995a} proved that, whenever $g$ is a strictly ergodic homeomorphism of $X$ (i.e.~$g$ is minimal and with a unique invariant Borel probability measure), the unique $g$-invariant measure is good. Akin \cite{Akin2005} established the converse implication: for any good measure $\mu$ on $X$, there exists a strictly ergodic homeomorphism $g$ of $X$ whose unique invariant Borel probability measure is $\mu$.

Dougherty, Mauldin and Yingst \cite{Dougherty2007} showed that $\mu_r$ is refinable if and only if $r$ is a root of an integer polynomial $P$ with $P(0)=\pm 1$, $P(1)=\pm 1$; $\mu_r$ is good if and only if it is refinable and $r$ has no algebraic conjugate in $]0,1[$. An interesting phenomenon is that when $\mu_r$ is refinable, $s \in \, ]0,1[$ is an algebraic conjugate of $r$ and $g$ is a homeomorphism of $X$ such that $g_* \mu_r=\mu_r$, then also $g_* \mu_s=\mu_s$. Given Akin's theorem, one wonders whether, for any refinable Bernoulli measure $\mu_r$, there exists a minimal homeomorphism $g$ of $X$ whose set of ergodic invariant Borel probability measures is equal to $\{\mu_s  \colon s \textrm{ is an algebraic conjugate of s in } \, ]0,1[\}$; this holds true, and that is one of the main results of \cite{Yingst2020}.

To prove this result, Yingst uses Baire category techniques. Given a compact, convex subset $K$ of the set of all Borel probability measures on $X$, denote 
\[G_K = \left\{ g \in \textrm{Homeo}(X) \colon \forall \mu \in K \ g_* \mu=\mu\right \} \] 
and let $G_K^*$ be the set of all $g \in G_K$ such that $g(U) \ne U$ for any nontrivial clopen set $U$. As pointed out in \cite{Yingst2020}, $G_K^*$ consists of the chain-transitive elements of $G_K$.

Then $G_K$ is a closed subgroup of the homeomorphism group of $X$, endowed with its unique Polish group topology; a neighborhood basis of the identity in $G_K$ is given by the subgroups 
\[ G_\mcA =\{g \in G_K \colon \forall A \in \mcA \ g(A)=A\} \]
where $\mcA$ ranges over all clopen partitions of $X$. Mapping a given clopen set to itself is an open condition in $G_K$, so $G_K^*$ is a closed subset of $G_K$; Yingst points out that $G_{\{\mu\}}^*$ is nonempty whenever $\mu$ is a refinable Bernoulli measure (this is true in greater generality, as we will see below). 

\begin{theorem*}[{\cite[Theorem~1.1]{Yingst2020}}]
Let $P$ be an irreducible integer polynomial with $P(0)\pm1$ and $P(1)=\pm 1$. let $R$ be the set of roots of $P$ which lie in $]0,1[$, and let $r \in R$.

Then $G_{\{\mu_r\}}^*$ is nonempty and for a generic element $g$ of $G_{\{\mu_r\}}^*$ the set of all $g$-invariant probability measures is equal to the closed convex hull of $\{\mu_s \colon s \in R\}$.

Similarly, assume that $\mu$ is a good measure on $X$. Then $G_{\{\mu\}}^*$ is nonempty and for a generic element $g$ of $G_{\{\mu\}}^*$, the set of all $g$-invariant probability measures is equal to $\{\mu\}$.
\end{theorem*}

The results above, for refinable Bernoulli measures and for good measures, have a common generalization. Before stating it, we need to introduce some more background. 

\begin{defn*}[see \cite{Ibarlucia2016} and \cite{Melleray2018}]
A compact, convex set $K$ of Borel probability measures on $X$ is a \emph{dynamical simplex} if:
\begin{itemize}
\item All elements of $K$ are full and atomless.
\item For any two $A,B \in \Clopen(X)$ such that $\mu(A)< \mu(B)$ for all $\mu \in K$, there exists a clopen $\tilde A \subset B$ such that $\mu(A)=\mu(\tilde A)$ for all $\mu \in K$.
\end{itemize}
\end{defn*}

It follows from a theorem of Glasner and Weiss \cite[Lemma~2.5]{Glasner1995a} that the set of invariant measures of any minimal homeomorphism of $X$ is a dynamical simplex.
When $K$ is a singleton $\{\mu\}$, $K$ is a dynamical simplex if, and only if, $\mu$ is a good measure. 

Assuming that $K$ has finitely many extreme points which are mutually singular, and $K$ is a dynamical simplex, Dahl \cite{Dahl2008} proved that there exists a minimal homeomorphism $g$ of $X$ whose invariant measures are exactly the elements of $K$ (the term \emph{dynamical simplex} was also introduced by Dahl, though the definition was not quite formulated as above). This extends Akin's theorem about good measures. It turns out that the assumption that extreme points of $K$ are mutually singular is redundant with the other assumptions, and that finite dimensionality may also be dispensed with. Indeed, it was proved by Ibarluc\'ia and the author in \cite{Ibarlucia2016} (with the additional assumption of \emph{approximate divisibility}, which was later shown to be superflous in \cite{Melleray2018} but will play a part in our arguments) that for any dynamical simplex $K$ there exists a minimal homeomorphism whose invariant measures are exactly the elements of $K$. This converse of the general result of Glasner and Weiss is thus an extension of Akin's converse which applies in the special case of uniquely ergodic systems.

Theorem 5.4 of \cite{Yingst2020} (closely related to earlier work of Dougherty, Mauldin and Yingst \cite{Dougherty2007}) shows that, whenever $P$ is an irreducible integer polynomial with $P(0)=\pm 1$, $P(1)=\pm 1$, $R$ is the set of roots of $P$ contained in $]0,1[$ and $r \in R$, the closed convex hull $K_r$ of all $\mu_s$, for $s \in R$, is a dynamical simplex; further, $G_{\{\mu_r\}}=G_{K_r}$ for all $r \in R$. Thus Yingst's theorem may be seen as asserting that, for certain dynamical simplices $K$, for a generic element $g \in G_K^*$ the set of all $g$-invariant Borel probability measures coincides with $K$. As an aside, note that any $g \in G_K \setminus G_K^*$, or more generally any $g \in G_K$ which is not minimal, must preserve some measure which does not have a full support, so the set of $g$-invariant Borel probability measures cannot be equal to $K$.

We prove here (Theorem \ref{t:generic_saturated} below) that the conclusion of Yingst's theorem holds for all dynamical simplices: whenever $K$ is a dynamical simplex, a generic element of $G_K^*$ has a set of invariant measures equal to $K$. We actually establish a more precise result, using the notion of a \emph{saturated} homeomorphism.

The strategy of \cite{Ibarlucia2016} is based on the notion of a \emph{Kakutani-Rokhlin partition}; we recall that, if $g$ is a homeomorphism of $X$, a Kakutani-Rokhlin partition for $g$ is a clopen partition $(A_{i,j})_{0 \le i \le n, 0 \le j \le n_i}$ such that for all $i$ and all $0\le j< n_i$ one has $g(A_{i,j})=A_{i,j+1}$. When $g$ is minimal, one can build a refining sequence of such partitions whose atoms generate the clopen Boolean algebra (see for instance \cite{Ibarlucia2016} for more details and further references), and then we can think of these partitions as giving better and better approximations of $g$. The proof given in \cite{Ibarlucia2016} that any dynamical simplex can be realized as the set of invariant measures of some minimal homeomorphism works by building a refining sequence of clopen partitions, which turn out to be Kakutani-Rokhlin partitions of the desired minimal homeomorphism $g$. These partitions, which we call $K$-partitions, also play a major part in our approach here (see the next section for their definition). To ensure that $g$ preserves $K$, we ask that $\mu(A_{i,j})=\mu(A_{i,j+1})$ for any $\mu \in K$ whenever $A_{i,j}$ is an element of one of the Kakutani-Rokhlin partitions for $g$ and $j<n_i$. One also needs to ensure that $g$ does not preserve measures which are not in $K$; to that end, one works with a stronger notion than minimality, namely with \emph{saturated} homeomorphisms. We skip the formal definition for now, see the next section for details. What matters here is that, 
 whenever $K$ is a dynamical simplex and $g$ is $K$-saturated, $g$ is minimal and the set of invariant measures of $g$ is equal to $K$; and the construction of \cite{Ibarlucia2016} produces a $K$-saturated element of $G_K$. 

One sees from their definition that $K$-saturated homeomorphisms form a $G_\delta$ subset of $G_K^*$. It is asked in \cite[Remark 2]{Ibarlucia2016} whether one can determine the closure in $G_K$ of the set of $K$-saturated homeomorphisms, and suggested that it might be equal to $G_K^*$, by analogy with \cite[Theorem~5.9]{Bezuglyi2006} which proves a similar result in $\Homeo(X)$. 
Our main result (proved in Section \ref{s:proof_main}) confirms this suspicion.

\begin{thm:saturated}
Let $K$ be a dynamical simplex. Then a generic element of $G_K^*$ is $K$-saturated. 
In particular, for a generic element $g$ of $G_K^*$ the set of $g$-invariant Borel probability measures is equal to $K$.
\end{thm:saturated}

This generalises Yingst's result for good measures and refinable Bernoulli measures to all dynamical simplices; the possibility that such a generalization may be true is mentioned at the end of \cite{Yingst2020}. 

Yingst mentions after the statement of his theorem that it \enquote{demonstrates a large class of uniquely or finitely ergodic homeomorphisms}; one then wonders how large that class is. Certainly, different dynamical simplices will yield non-conjugate uniquely or finitely ergodic homeomorphisms (when there is just one, or finitely many, extreme points), but it is not clear a priori whether two generic elements of $G_K^*$ are conjugate. Thus we investigate when there exist comeager conjugacy classes in $G_K^*$ and establish the following result.

\begin{thm:comeager}
There exists a comeager conjugacy class in $G_K^*$ if, and only if $K$ is a singleton $\{\mu\}$, and $\mu$ is the unique invariant probability measure for some odometer $g$. In that case, the conjugacy class of $g$ is comeager in $G_K^*$.
\end{thm:comeager}

\noindent \emph{Acknowledgements.} I am grateful to A. Yingst for a careful reading of an earlier version of this article, as well as many valuable suggestions which helped improve the exposition. Thanks are also due to the anonymous referee for several useful corrections and suggestions.

\section{Some background on dynamical simplices}
Our  argument is based on some terminology and results from \cite{Ibarlucia2016} and \cite{Melleray2018}, which we recall now (what we call "$K$-partition" here is called "KR-partition" there). It probably helps to have some familarity with Kakutani-Rokhlin partitions and the methods used in the aforementioned papers to follow the arguments, but we recall everything we need from these articles in order to establish our results. 

We recall that a dynamical simplex is a compact, convex set $K$ of probability measures on $X$ such that all elements of $K$ are atomless and have full support, and whenever $A,B \in \Clopen(X)$ are such that $\mu(A)< \mu(B)$ for all $\mu \in K$ there exists a clopen $\tilde A \subset B$ such that $\mu(A)=\mu(\tilde A)$ for all $\mu \in K$. 

We fix a dynamical simplex $K$.
Let 
\[G_K=\{g \in \Homeo(X) \colon \forall \mu \in K \ g_* \mu=\mu\} \]
where as usual $g_* \mu(A)=\mu(g^{-1}A)$ for any Borel $A$. We endow $G_K$ with the topology whose base of open neighborhhods of $1$ is given by sets of the form 
\[ \{g \in G_K \colon \forall A \in \mcA \ g(A)=A\}\]
where $\mcA$ ranges over all clopen partitions of $X$. This turns $G_K$ into a Polish group.

\begin{lemma}[{\cite[Proposition~2.7]{Ibarlucia2016}}]
Assume that $A,B$ are clopen and $\mu(A)=\mu(B)$ for all $\mu \in K$. Then there exists $g \in G_K$ such that $g(A)=B$. 
\end{lemma}
This property is key to our constructions. If $\mu(A)< \mu(B)$ for all $\mu \in K$, then it follows from this fact and the definition of a dynamical simplex that there exists $g \in G_K$ such that $g(A) \subset B$.

Basic compactness arguments enable one to prove the following facts.

\begin{lemma}[{\cite[Proposition~2.5]{Ibarlucia2016}}]\label{l:intersection} \ 
\begin{itemize}
\item For any nonempty $A \in \Clopen(X)$, $\inf \{\mu(A) \colon  \mu \in K\} >0$. 

\item[•] If $d$ is any compatible metric, and $\varepsilon >0$, there exists some $\delta>0$ such that whenever $B$ has diameter less than $\delta$, one has $\mu(B) \le \varepsilon$ for all $\mu \in K$.
\end{itemize}
\end{lemma}

Below, we will make constant use of the homogeneity properties of dynamical simplices; the following lemma will be used in the proof of our main result (Theorem \ref{t:generic_saturated}). We prove it now in the hope of helping the reader get acquainted with dynamical simplices.

\begin{lemma}\label{l:toyexample}
Let $U, V$ be clopen and such that $\mu(U)=\mu(V)$ for all $\mu \in K$. Let $\mcA$ be a clopen partition of $U$, and $\mcB$ be a clopen partition of $V$. Let also $C_1,\ldots,C_n$, $D_1,\ldots,D_m$ be nonempty clopen subsets of $U,V$ respectively. 

There exist clopen subsets $Y_1,\ldots,Y_q$ of $U$, and clopen subsets $Z_1,\ldots,Z_q$ of $V$, such that:
\begin{itemize}
\item $\forall i \ \forall \mu \in K \quad \mu(Y_i)= \mu(Z_i)$.
\item $Y_1,\ldots,Y_q$ form a partition of $U$ which refines $\mcA$, and $Z_1,\ldots,Z_q$ form a partition of $V$ which refines $\mcB$.
\item For all $i \in \{1,\ldots,n\}$ and all $j \in \{1,\ldots,m\}$ there exists $k \in \{1,\ldots q\}$ such that $Y_k \subseteq C_i$ and $Z_k \subseteq D_j$.
\end{itemize}

\end{lemma}

\begin{proof}
Fix a compatible distance on $X$. Reducing $C_1,\ldots,C_n$ if necessary, we may assume that they are pairwise disjoint, and that each $C_i$ is contained in some element of $\mcA$; and similarly, mutatis mutandis, for $D_1,\ldots,D_m$ and $\mcB$.

Fix $\varepsilon >0$ for which any clopen subset $U$ of $X$ of diameter less than $ \varepsilon$ is such that $n \mu(U) < \mu(D_j)$ for all $j$. Inside each $C_i$, we may find $m$ nonempty, disjoint clopen subsets $(A_{i,j})_{1 \le j \le m}$ of diameter less than $\varepsilon$. For every $j \in \{1,\ldots,m\}$ and every $\mu \in K$, we then have
\[\sum_{i=1}^n \mu(A_{i,j}) < \mu(D_j) \] 

Since $K$ is a dynamical simplex, we may then find disjoint clopen subsets $B_{i,j}$ inside each $D_j$ such that $\mu(B_{i,j})=\mu(A_{i,j})$ for all $i \in \{1,\ldots,n\}$.

Using a bijection from $\{1,\ldots,n\} \times \{1,\ldots,m\}$ to $\{1,\ldots,nm\}$, we enumerate $A_{i,j}$ under the form $Y_1,\ldots,Y_{nm}$, and $B_{i,j}$ as $Z_1,\ldots,Z_{nm}$. Then $\mu(Y_k)=\mu(Z_k)$ for all $k$, and for all $i,j$ there exists $k$ such that $Y_k \subseteq C_i$ and $Z_k \subseteq D_j$. 

Let $\tilde U = U \setminus \bigsqcup_{k=1}^{nm} Y_k$, $\tilde V= V \setminus \bigsqcup_{k=1}^{nm} Z_k$. Then $\mu(\tilde U)=\mu(\tilde V)$ for all $\mu \in K$, so there exists choose $g \in G_K$ mapping $\tilde U$ to $\tilde V$. Note that $\mcA$ induces a clopen partition $\tilde A$ of $\tilde U$ (obtained by intersecting elements of $\mcA$ with $\tilde U$), and $\mcB$ similarly induces a clopen partition $\tilde \mcB$ of $\tilde V$. 
Consider the coarsest partition of $\tilde U$ refining both $\tilde \mcA$ and $g^{-1}(\tilde \mcB)$, and list its elements as $Y_{nm+1},\ldots,Y_q$. Then set $Z_k=g(Y_k)$ for all $k \in \{nm+1,\ldots,q\}$ to complete the construction.
\end{proof}


%
%
%

We need an additional property of dynamical simplices.
\begin{theorem}[{\cite[Corollary 2.6]{Melleray2018}}]
$K$ is \emph{approximately divisible}, i.e.~for any clopen $A$, any $\varepsilon >0$ and any positive integer $n$ there exists a clopen $B \subseteq A$ such that 
\[\forall \mu \in K \quad \mu(A)-\varepsilon \le n \mu(B) \le \mu(A) \] 
\end{theorem}

The way we are going to use approximate divisibility is (using the same notations as above) by applying the defining property of a dynamical simplex to find disjoint clopen $B_i \subseteq A$ such that 

\[ \forall \mu \in K \ \forall i  \in \{1,\ldots,n\} \quad \mu(B_i)=\mu(B_1) \quad \text{and} \quad \mu\left(A \setminus \bigsqcup_{i=1}^n B_i \right) \le \varepsilon \] 

We recall the following standard definition.

\begin{defn}
For $g \in \Homeo(X)$, a \emph{Kakutani-Rokhlin partition} for $g$ is a clopen partition $\mcA=(A_{i,j})_{i \in I_\mcA, 0\le j \le n_i}$ such that $A_{i,j+1}=g(A_{i,j})$ whenever $0 \le j < n_i$.
\end{defn}

We will need an abstract version of these partitions (intuitively, we manipulate clopen partitions that could be Kakutani-Rokhlin partitions for some homeomorphism, and use those to build homeomorphisms).

\begin{defn} A \emph{$K$-partition} is a clopen partition $\mcA=(A_{i,j})_{i \in I_\mcA, 0\le j \le n_i}$
such that 
\[ \forall i \in I_\mcA \  \forall j,k \le n_i \ \forall \mu \in K \quad \mu(A_{i,j})= \mu(A_{i,k})\] 

We use the notation $F_\mcA$ to denote $\{(i,j) \colon i \in I_\mcA \ , 0 \le j \le n_i\}$.

We call a set $(A_{i,j})_{0\le j \le n_i}$ a \emph{column} of the partition; we say that $n_i+1$ is the \emph{height} of this column, $A_{i,0}$ is its \emph{base} and $A_{i,n_i}$ is its \emph{top}.
Similarly, the union of all $A_{i,0}$ is the \emph{base} of the partition, and the union of all $A_{i,n_i}$ is its \emph{top}. We sometimes call $A_{i,j+1}$ the \emph{successor} of $A_{i,j}$ in $\mcA$ (for $j< n_i$).

We say that a homeomorphism $g$ is \emph{compatible} with $\mcA$ if $g(A_{i,j})=A_{i,j+1}$ for all $i$ and all $j < n_i(\mcA)$. 
\end{defn}

Note that if $g$ is compatible with $\mcA$ then $g(\topp(\mcA))=\base(\mcA)$; also, given $\mcA$, the set of all $g \in G_K$ which are compatible with $\mcA$ is clopen. When $g$ is a homeomorphism such that $g_* \mu = \mu$ for all $\mu \in K$, any Kakutani-Rokhlin partition for $g$ is a $K$-partition with which $g$ is compatible. 

\begin{defn}
A $K$-partition $\mcB$ \emph{refines} another $K$-partition $\mcA$ if:

\begin{enumerate}
\item For all $(i,j) \in F_\mcB$ there exists $(k,l) \in F_\mcA$ such that $B_{i,j} \subseteq A_{k,l}$. \\
(i.e.~every element of $\mcA$ is a union of elements of $\mcB$) 
\item For all $i \in I_\mcB$, there exists $k \in I_\mcA$ such that $B_{i,0} \subseteq A_{k,0}$.\\
 (i.e.~the base of $\mcB$ is contained in the base of $\mcA$)
\item For all $((i,j),(k,l)) \in F_\mcB \times F_\mcA$, if $B_{i,j} \subseteq A_{k,l}$ and $l<n_k(\mcA)$, then $j<n_i(\mcB)$ and $B_{i,j+1} \subseteq A_{k,l+1}$.\\
 (i.e.~the successor of $B_{i,j}$ is contained in the successor of $A_{k,l}$)
\item For all $(i,j) \in F_\mcB$, if $B_{i,j} \subseteq A_{k,n_k(\mcA)}$ for some $k$ and $j< n_i(\mcB)$, then $B_{i,j+1}$ is contained in some $A_{p,0}$.\\
 (i.e.~if $B_{i,j}$ is contained in the top of $\mcA$ and $j< n_i$ then the successor of $B_{i,j}$ is contained in the base of $\mcA$)
\end{enumerate}
Informally, the columns of $\mcB$ have been obtained by cutting the columns of $\mcA$ vertically, and stacking these fragments of columns on top of each other.

Whenever $B_{i,j}$ is contained in $A_{k,0}$ we say that $(B_{i,j},\ldots,B_{i,j+n_i(\mcA)})$ is a \emph{copy} of the column $(A_{k,0},\ldots,A_{k,n_k(\mcA)})$ contained in $(B_{i,0},\ldots,B_{i,n_k(\mcB)})$. 
\end{defn}

\begin{defn}
Let $K$ be a dynamical simplex, and $g \in G_K$. We say that $g$ is $K$-\emph{saturated} if, whenever $A,B \in \Clopen(X)$  are such that $\mu(A)=\mu(B)$ for all $\mu \in K$, there exists a Kakutani-Rokhlin partition $(A_{i,j})_{i \in I, 0 \le j \le n_i}$ for $g$ such that $A,B$ are unions of atoms of this partition and for all $i$ 
\[\left|\{j  \colon A_{i,j} \subseteq A \} \right| = \left|\{j \colon A_{i,j} \subseteq B \} \right| \]
\end{defn}
(here and throughout the article $|F|$ denote the cardinality of a finite set $F$).

We note that, for a given $A,B \in \Clopen(X)$ such that $\mu(A)=\mu(B)$ for all $\mu \in K$, the set of all $g$ admitting a Kakutani-Rokhlin partition with the above property is clopen subset in $G_K$: if $\mcA$ is such a partition for $g$, any $h \in G_k$ which is compatible with $\mcA$ will also have the desired property. Since there are countably many clopen subsets of $X$, it follows that $K$-saturated homeomorphisms form a $G_\delta$ subset of $G_K$.

The following result explains why saturated homeomorphisms play such an important role in our arguments.

\begin{prop}\cite[Corollary~4.3]{Ibarlucia2016}
Assume that $g \in G_K$ is saturated. Then the set of $g$-invariant Borel probability measures is equal to $K$.
\end{prop}

Combined with the above fact, the following proposition is the heart of the proof of the main result of \cite{Ibarlucia2016}.

\begin{prop}\label{prop:fundamental}
Let $\mcA$ be a $K$-partition; there exists a $K$-saturated $g \in G_K$ which is compatible with $\mcA$.
\end{prop}

This proposition is not formally stated in \cite{Ibarlucia2016} (though it is implicit in Remark $2$ of that paper) but is established as in the proof of \cite[Proposition~3.6]{Ibarlucia2016} by an inductive argument based on the following lemma.

\begin{lemma}[{combination of \cite[Propositions~2.5 and~3.5]{Ibarlucia2016}}]\label{lemma:balanced}
Let $\mcA$ be a $K$-partition. Assume that $U,V$ are clopen and $\mu(U)=\mu(V)$ for all $\mu \in K$. Let $W$ be a nonempty clopen subset of $\topp(\mcA)$.

Then there exists a $K$-partition $\mcB$ refining $\mcA$, whose top is contained in $W$, such that $U,V$ are both unions of atoms of $\mcB$ and for all $k$ one has
\[\left|\{l \in \{0,\ldots,n_k(\mcB) \} \colon B_{k,l} \subseteq U \} \right| = \left|\{l \in \{0,\ldots,n_k(\mcB) \} \colon B_{k,l} \subseteq V \} \right| \]

When $U,V$ satisfy the conditions above, we say that $(U,V)$ are $\mcB$-equivalent.
\end{lemma}
\section{Proof that a generic element of $G_{K}^*$ is $K$-saturated}\label{s:proof_main}
Throughout this section we fix a dynamical simplex $K$. We recall that $G_K$ is the group of all homeomorphisms which preserve $K$, and $G_K^*$ is the set of all $g \in G_K$ such that, for any nontrivial clopen set $U$, one has $g(U) \ne U$. Our aim is to prove the following theorem. 

\begin{theorem} \label{t:generic_saturated}
Let $K$ be a dynamical simplex. Then a generic element of $G_K^*$ is $K$-saturated.
\end{theorem}

As pointed out above, $K$-saturated homeomorphisms form a $G_\delta$ subset of $G_K^*$ (and we know that there exist $K$-saturated homeomorphisms by the main result of \cite{Ibarlucia2016}, so $G_K^*$ is also nonempty). So our aim really is to prove that the set of $K$-saturated homeomorphisms is dense in $G_{K}^*$; the remainder of this section is devoted to the proof of that fact. 


%

We fix $\varphi \in G_K^*$, and a clopen partition $\mcA$ of $X$. Our goal is to prove that there exists a $K$-saturated element $\psi \in G_K$ such that $\psi(\alpha)=\varphi(\alpha)$ for all $\alpha \in \mcA$.

The proof is going to proceed by building certain $K$-partitions, and we think of $\mcA$ as a $K$-partition whose columns all have height equal to $1$.

\begin{defn}
Let $\mcB$ be a $K$-partition which refines $\mcA$. For every atom $B_{i,j}$ of $\mcB$ there exists an atom $\alpha(i,j) $ of $\mcA$ such that $B_{i,j} \subseteq \alpha(i,j)$; we say that $\mcB$ \emph{respects $\varphi$} if $B_{i,j+1} \subseteq \varphi(\alpha(i,j))$ whenever $j< n_i$.
\end{defn}

\begin{lemma}\label{l:compatiblepartition}
Assume that $\mcB$ refines $\mcA$ and respects $\varphi$, and let $g \in G_K$ be compatible with $\mcB$. Then
\begin{enumerate}
\item\label{eq:compatible1} For every $\alpha \in \mcA$ we have $g(\alpha \setminus \topp(B))= \varphi(\alpha) \setminus \base(\mcB)$. 
\item\label{eq:compatible2} For every $\alpha \in \mcA$, and every $\mu \in K$, $\mu(\alpha \cap \topp(B))= \mu( \varphi(\alpha) \cap \base(\mcB))$.
\item\label{eq:compatible3} For every $\alpha \in \mcA$ which does not intersect $\topp(\mcB)$ we have $g(\alpha)=\varphi(\alpha)$.
\item\label{eq:compatible4} $ \base(\mcB) \subseteq \bigcup \{\varphi(\alpha) \colon \alpha \in \mcA \text{ and } \alpha \cap \topp(\mcB) \ne \emptyset\}$.
\end{enumerate}
\end{lemma}

\begin{proof}
We begin by proving \eqref{eq:compatible1}. Let $\alpha$ be an atom of $\mcA$. 
Since $\mcB$ refines $\mcA$, respects $\varphi$ and $g$ is compatible with $\mcB$, we have $g(\alpha \setminus \topp(\mcB)) \subseteq \varphi(\alpha)$; further, $g(\alpha \setminus \topp(\mcB))$ is disjoint from $\base(\mcB)$. This establishes one inclusion.  Also, every atom of $\mcB$ which is not contained in $\base(\mcB)$ is contained in $\varphi(\beta)$ for some atom $\beta$ of $\mcA$. Thus, $\varphi(\alpha) \setminus \base(\mcB)$ is a union of atoms of $\mcB$; these atoms may be listed as $(B_{i,j})_{(i,j )\in F}$, and for all $(i,j) \in F$ we have $j \ge 1$. Using again the fact that $\mcB$ respects $\varphi$, each $B_{i,j-1}$ is contained in $\alpha \setminus \topp(\mcB)$, and $B_{i,j}=g(B_{i,j-1})$, so $\varphi(\alpha) \setminus \base(\mcB)$ is contained in $g(\alpha \setminus \topp(\mcB))$.

Now that \eqref{eq:compatible1} is proved, fix $\mu \in K$ and $\alpha$ an atom of $\mcA$. We have both 
\begin{eqnarray*}
&\mu(\alpha)= \mu(\alpha \setminus \topp(B)) + \mu(\alpha \cap \topp(\mcB)) \quad \text{and} \\
&\mu(\varphi(\alpha))= \mu(\varphi(\alpha)\setminus \base(\mcB))+ \mu(\varphi(\alpha) \cap \base(\mcB))
\end{eqnarray*}
Since $g$ and $\varphi$ preserve $\mu$ , \eqref{eq:compatible1} implies that $\mu(\alpha \cap \topp(\mcB))= \mu(\varphi(\alpha) \cap \base(\mcB))$.

The third condition follows immediately: if $\alpha \cap \topp(\mcB)= \emptyset$, then $\varphi(\alpha) \cap \base(\mcB)$ must also be empty by the second condition, and then \eqref{eq:compatible1} gives $g(\alpha)=\varphi(\alpha)$.

Since $\varphi(\alpha)$ does not intersect $\base(\mcB)$ if $\alpha$ does not intersect $\topp(\mcB)$, we also obtain that $\base(\mcB)$ is contained in the union of all $\varphi(\alpha)$ such that $\alpha \cap \topp(\mcB) \ne \emptyset$, and that proves \eqref{eq:compatible4}.
%
\end{proof}

\begin{defn}
Let $\mcB$ be a $K$-partition which refines $\mcA$ and respects $\varphi$. We denote by $N(\mcB)$ the number of atoms of $\mcA$ which intersect $\topp(\mcB)$.
\end{defn}

Our goal is to prove that there exists $\mcB$ which refines $\mcA$, respects $\varphi$ and such that $N(\mcB)=1$. Indeed, we note the following fact.

\begin{lemma}\label{l:parcequecestnotreprojet}
Assume that there exists a $K$-partition $\mcB$ which refines $\mcA$, respects $\varphi$, and is such that $N(\mcB)=1$. Then there exists a $K$-saturated $\psi \in G_K$ such that $\psi(\alpha)=\varphi(\alpha)$ for all $\alpha \in \mcA$.
\end{lemma}

\begin{proof}
Let $\mcB$ be such a $K$-partition. By Proposition \ref{prop:fundamental} there exists a saturated $\psi$ which is compatible with $\mcB$. 

For every $\alpha \in \mcA$ which does not intersect $\topp(\mcB)$, we have $\varphi(\alpha) = \psi(\alpha)$ by Lemma \ref{l:compatiblepartition}\eqref{eq:compatible3}. Since $\varphi$ and $\psi$ are bijections, and there exists a unique atom $\beta$ of $\mcA$ which intersects $\topp(\mcB)$, this atom must also satisfy $\varphi(\beta)=\psi(\beta)$, so we obtain as promised that $\varphi(\alpha)=\psi(\alpha)$ for every $\alpha \in \mcA$.
\end{proof}

Assume that $\mcB$ refines $\mcA$ and respects $\varphi$. Let $C$ be a column of $\mcB$; $\topp(C)$ is contained in some atom of $\mcA$, which we denote by $\tau(C)$. 

We endow the set of columns $\Coll(\mcB)$ with a directed graph structure $\Gamma_\mcB$ (possibly, with loops), by declaring that 
\[ \left( (C,D) \in \Gamma_\mcB \right) \Leftrightarrow \left( \varphi(\tau(C)) \cap \base(D) \ne \emptyset \right)\]

Note that above we are manipulating $\varphi(\tau(C))$ and not $\varphi(C)$; that is because we are looking for a $K$-saturated element which coincides with $\varphi$ on $\mcA$, so the conditions we have to satisfy are imposed by the behaviour of $\varphi$ on atoms of $\mcA$, not on smaller clopen sets (all the information we can use is what $\varphi$ does to atoms of $\mcA$, and the fact that this behavior on $\mcA$ is compatible with belonging to $G_K^*$).

\begin{defn}
Let $\mcB$ be a $K$-partition. We say that $\mcB$ is \emph{admissible} if:
\begin{itemize}
\item[•] $\mcB$ refines $\mcA$ and respects $\varphi$.
\item For every atom $\alpha$ of $\mcA$, $\varphi(\alpha)$ is a union of atoms of $\mcB$.
\item[•] For any $C,D \in \Coll(\mcB)$ there exists a path in $\Gamma_\mcB$ starting at $C$ and ending at $D$. 
\end{itemize}
\end{defn}

The second condition above mildy simplifies the argument (it helps minimize the number of cuts we have to do during the construction); the third condition is key to our construction.

\begin{lemma}
There exists an admissible $K$-partition.
\end{lemma}

\begin{proof}
Let $\mcA'$ be the coarsest partition which refines both $\mcA$ and $\{\varphi(\alpha) \colon \alpha \in \mcA\}$. We view $\mcA'$ as a $K$-partition where each column has height $1$. It satisfies the first two conditions of the definition of an admissible $K$-partition.

Fix an atom $\beta$ of $\mcA'$. Let $U$ be the union of atoms of $\mcA'$ which one can reach from $\beta$ by following a path in $\Gamma_{\mcA'}$. Then $U$ is clopen and $\varphi(U) \subseteq U$, whence $U=X$ since $\varphi \in G_K^*$. This proves that $\mcA'$ is admissible.
\end{proof}

Given Lemma \ref{l:parcequecestnotreprojet}, our proof will  be complete as soon as we manage to establish the validity of the following lemma; the argument is based on repeated applications of the cutting and stacking procedure.

\begin{lemma}
Assume that $\mcB$ is an admissible $K$-partition, and $N(\mcB)>1$. Then there exists an admissible $K$-partition $C$ such that $N(\mcC)<N(\mcB)$.
\end{lemma}

\begin{proof}
We fix $\mcB$, and let $\alpha_1,\ldots,\alpha_N$ denote the atoms of $\mcA$ which intersect $\topp(\mcB)$. 

Recall that, for every column $C$ of $\mcB$, there exists an atom $\tau(C)$ of $\mcA$ such that $\topp(C) \subseteq \tau(C)$; and an atom $\beta(C)$ of $\mcA$ such that $\base(C) \subseteq \varphi(\beta(C))$. By definition, $(C,D) \in \Gamma_\mcB$ if and only if $\tau(C)=\beta(D)$.

Pick a column $C$ such that $\beta(C)\ne \alpha_1$ (the existence of such a column follows from Lemma \ref{l:compatiblepartition}\eqref{eq:compatible2}). Since $\mcB$ is admissible, there exists a path $C=C_0,\ldots,C_p$ such that $\tau(C_p)=\alpha_1$; fix such a path with $p$ minimal. If $\beta(C_p)=\alpha_1$ then either $p=0$ or $\tau(C_{p-1})=\alpha_1$ by definition of $\Gamma_\mcB$; the second possibility contradicts the minimality of $p$. So either 
$C_0$ or $C_p$ gives us a column $C$ such that $\beta(C) \ne \alpha_1$ and $\tau(C)=\alpha_1$. 

We write $\base(C)=U \sqcup V$, where $U$, $V$ are nonempty clopen, and cut $C$ into two finer columns with base $U,V$. Explicitly, assume that $C$ is of the form $(B_{i,0},\ldots,B_{i,n_i})$; for each $j \in \{1,\ldots,n_i\}$ we choose a clopen subset $U_j$ of $B_{i,j}$ such that $\mu(U_j)=\mu(U)$ for all $\mu \in K$, and let $V_j=B_{i,j} \setminus U_j$. Then we form a new partition, where $C$ is replaced by two finer columns, one being $(U,U_1,\ldots,U_{n_i})$ and the other being $(V,V_1,\ldots,V_{n_i})$. The partition we obtain refines $\mcB$ and respects $\varphi$; below, whenever we mention cutting, we apply a similar procedure.

We set aside one of the two columns we just built and name it $\tilde C$.

Considering a path $(D_0,\ldots,D_q)$ of minimal length such that $\beta(D_0)= \alpha_1$ and $\tau(D_q) \ne \alpha_1$, we similarly obtain the existence of a column $D$ such that $\beta(D)=\alpha_1$ and $\tau(D) \ne \alpha_1$. As above, we cut $D$ into two nontrivial subcolumns and call one of those $\tilde D$. We let $\mcB'$ denote the partition that we produced by cutting $C,D$ into two subcolumns each.

There may exist columns in $\mcB'$ such that $\beta(E)=\alpha(E)=\alpha_1$; we wish to reduce to the case where there are no such columns. In case no such column exists, we may directly skip to the next step of the construction. Otherwise, let $k \ge 1$ denote the number of these columns.
For every such $E$, we choose a large integer $l_E$ and use approximate divisibility to write
$$\base(E) = \left(\bigsqcup_{i=1}^{l_E} E_i \right) \sqcup F $$
where for all $\mu \in K$ one has
\begin{eqnarray}
&\forall \mu \in K \ \forall i,j \quad \mu(E_i)=\mu(E_j) \quad \text{ and } \\
&\forall \mu \in K \quad \mu(E_1)+ \mu(F) < \frac{1}{k} \min(\mu(\base(\tilde C)), \mu(\base(\tilde D))) \label{eq:approx}
\end{eqnarray}
We then form a new partition refining $\mcB'$, in which each column $E$ with $\beta(E)=\tau(E)=\alpha_1$ is replaced by two new columns $\tilde E_1$ and $\tilde E_2$, with $\tilde E_1$ being obtained by stacking $l_E$ copies of $E$ with base $E_i$ on top of each other, and $\tilde E_2$ being the remainder (namely, a copy of $E$ with base $F$). Since each of the columns we stacked on top of each other has a base contained in $\varphi(\alpha_1)$ and a top contained in $\alpha_1$, the partition we obtain respects $\varphi$.

The sum of measures of the bases of all columns $\tilde E_1$, $\tilde E_2$ is strictly less than $\mu(\base (\tilde C))$ as well as $\mu(\base(\tilde D))$, for all $\mu \in K$; this is what we gained by our use of approximate divisibility in \eqref{eq:approx}. Using the fact that $K$ is a dynamical simplex, and cutting $\tilde C$, $\tilde D$ as necessary, we may then form new columns by stacking each $\tilde E_1$ on top of a copy of $\tilde C$, and then a copy of $\tilde D$ on top of that; and similarly for each $\tilde E_2$. 

Intuitively, we have first used approximate divisibility to make the columns with $\beta(E)=\tau(E)=\alpha_1$ very thin, and then we have stacked them in the middle of some new columns.
We have thus formed a new $K$-partition $\mcB''$ which refines $\mcB'$ and respects $\varphi$, since $\tau(\tilde C) = \beta(\tilde D)=\alpha_1$. 

Every column $F$ of $\mcB''$ is of one of the following three types:
\begin{itemize}
\item $\beta(F)=\alpha_1$ and $\tau(F) \ne \alpha_1$ (type $1$);
\item $\beta(F) \ne \alpha_1$ and $\tau(F)=\alpha_1$ (type $2$); 
\item $\beta(F) \ne \alpha_1$ and $\tau(F)\ne\alpha_1$ (type $3$).
\end{itemize}

Let $U_1,\ldots,U_n$ denote the tops of the columns of type $1$, and $V_1,\ldots,V_m$ denote the bases of the columns of type $2$. Lemma \ref{l:compatiblepartition}\eqref{eq:compatible2} implies that 
\[\forall \mu \in K \quad \mu\left(\bigsqcup_{i=1}^n U_i \right) =\mu\left(\bigsqcup_{j=1}^m V_j \right) \]

Using Lemma \ref{l:toyexample}, we may find clopen subsets $Y_1,\ldots,Y_q$, $Z_1,\ldots,Z_q$ such that
\begin{itemize}
\item For all $\mu \in K$ and for all $i$ one has $\mu(Y_i)=\mu(Z_i)$;
\item $Y_1,\ldots,Y_q$ partition $\bigsqcup_{i=1}^n U_i$, and $Z_1,\ldots,Z_q$ partition $\bigsqcup_{j=1}^m V_j$;
\item For all $i$ there exists $j,k$ such that $Y_i \subseteq U_j$ and $Z_i \subseteq V_k$;
\item For all $j,k$, if there exists a column $G$ of type $1$ such that $\tau(G) = \alpha_j$ and a column $H$ of type $2$ such that $\beta(H)=\alpha_k$, then there exists $i$ such that $Y_i \subseteq \alpha_j$ and $Z_i \subseteq \varphi(\alpha_k)$.
\end{itemize}

We then cut some columns of $\mcB''$ so as to form, for each $i \in \{1,\ldots,q\}$, a column with top $Y_i$ and another with base $Z_i$; and finally form a $K$-partition $\mcC$ refining $\mcB''$ by stacking each new column with top equal to $Y_i$ on top of the new column with base equal to $Z_i$. 

By construction $\mcC$ refines $\mcB$; $\mcC$ is compatible with $\varphi$; for every atom $\alpha$ of $\mcA$ $\varphi(\alpha)$ is a union of atoms of $\mcC$; and $\topp(\mcC) \cap \alpha_1 = \emptyset$ so $N(\mcC) \le N(\mcB)-1$.
It remains to prove that $\mcC$ is admissible (note that the admissibility of $\mcB$ was key to our construction, since it enabled us to produce the columns $\tilde C$, $\tilde D$, which in turn allowed us to get rid of all the columns whose top is contained in $\alpha_1$). 

During the construction, we ensured that the following properties hold: 
\begin{itemize}
\item[•] Any column $A$ of $\mcB$ such that $\tau(A) \ne \alpha_1$ and $\beta(A) \ne \alpha_1$ is also a column of $\mcC$ (these columns have not been modified).
\item[•] Given any $j,k \ge 2$ and columns $A,B$ of $\mcB$ such that $\beta(A)=\alpha_j$, $\tau(A)=\alpha_1$, $\beta(B)=\alpha_1$, $\tau(B)=\alpha_k$, there exists a column $E$ of $\mcC$ such that $\beta(E)=\alpha_j$ and $\tau(E)=\alpha_k$. We guaranteed this by our use of Lemma \ref{l:toyexample}, and by cutting $C,D$ in two pieces at the beginning so that nothing had been lost before applying Lemma \ref{l:toyexample}.

\end{itemize}
Now, let $A$, $B$ be two columns of $\mcC$. Since $\mcB$ is admissible, there exist columns $C_0,\ldots,C_{n}$ of $\mcB$  such that
\begin{itemize}
\item $\tau(C_0)= \tau(A)$ and $\beta(C_n)=\beta(B)$.
\item For all $i\le n-1$, $\tau(C_i)= \beta(C_{i+1})$.
\end{itemize}
If $\tau(C_i) \ne \alpha_1$ for each $i$, then  $A,C_1,\ldots,C_{n-1},B$ form a path joining $A$ and $B$ in $\Gamma_\mcC$.
Else, there exists some $i$ such that $\tau(C_i)= \alpha_1= \beta(C_{i+1})$; and (considering the shortest possible path) we may as well assume that there is a unique such $i$. By construction, there exists a column $D$ of $\mcC$ such that $\beta(D)= \beta(C_i)$ and $\tau(D)= \tau(C_{i+1})$. Then the path $A,C_1,\ldots,C_{i-1},D,C_{i+2},\ldots,B$ joins $A$ and $B$ in $\Gamma_\mcC$. 
\end{proof}

\section{Comeager conjugacy classes and odometers}

We again fix a dynamical simplex $K$; given a $K$-partition $\mcA$, we consider 
\begin{eqnarray*} O_\mcA &=& \{ g \in G_K^* \colon g \textrm{ is compatible with } \mcA \} \\
G_\mcA &=& \{ g \in G_K \colon \forall A \in \mcA \ g(A)=A\} 
\end{eqnarray*}

The sets $O_\mcA$ form a basis of the topology of $G_K^*$; the subgroups $G_\mcA$ form a basis of neighborhoods of the identity in $G_K$.

In our context, we may formulate Rosendal's criterion for the existence of comeager conjugacy classes for a Polish group action as follows (for a proof, see \cite[Proposition~3.2]{BenYaacov2017}).

\begin{prop}\label{prop:criterion}
There exists a comeager conjucacy class in $G_K^*$ if, and only if
\begin{enumerate}
\item The action of $G_K$ on $G_K^*$ by conjugation is topologically transitive .
\item \label{item:WAP} For any $K$-partition $\mcA$, there exists a $K$-partition $\mcB$ refining $\mcA$ such that, for any $K$-partitions $\mcC$, $\mcD$ refining $\mcB$ there exists $g \in G_\mcA$ such that $g O_\mcC g^{-1} \cap O_\mcD \ne \emptyset$.
\end{enumerate}
\end{prop}

We recall that a continuous action of a Polish group $G$ on a Polish space $X$ is \emph{topologically transitive} if it admits a dense orbit or, equivalently in this context, if for any two nonempty open subsets $U,V$ of $X$ there exists $g \in G$ such that $g(U) \cap V \ne \emptyset$.

The first part of the above criterion turns out to be always satisfied.

\begin{prop}\label{prop:saturated}
The conjugacy class of any $K$-saturated element is dense in $G_K^*$.\end{prop}

\begin{proof}
Fix a $K$-saturated $g$ and a $K$-partition $\mcA$. By saturation of $g$, there exists a Kakutani-Rokhlin partition $\mcB=(B_{i,j})$ for $g$ such that any two atoms of $\mcA$ belonging to the same column of $\mcA$ are $\mcB$-equivalent (see Lemma \ref{lemma:balanced} for the definition). Thus, up to a reordering of the atoms within each of its columns, $\mcB$ is a refinement of $\mcA$; in other words, there  exists for all $i$ a bijection $\sigma_i$ of $\{0,\ldots,n_i(\mcB)\}$ such that $(B_{i,\sigma_i(j)})$ refines $\mcA$.

We may find $h \in G_K$ such that for all $i$ and all $j \in \{0,\ldots,n_i(\mcB)\}$ one has $h(B_{i,j})=B_{i,\sigma_i(j)}$. Then $(B_{i,\sigma_i(j)})$ is a Kakutani-Rokhlin partition for $hgh^{-1}$, which proves that $hgh^{-1}$ belongs to $O_\mcA$.
\end{proof}

Our aim is now to prove that the second condition of Proposition \ref{prop:criterion} is satisfied exactly when $K$ is a singleton $\{\mu\}$, and $\mu$ is the unique invariant measure of an odometer.

Given integers $n_1,\ldots,n_k \ge 0$, with at least one being different from $0$, we set
$$r(n_1,\ldots,n_k)= \frac{1}{\sum_{i=1}^k n_i} (n_1,\ldots,n_k) $$
Let $d=\textrm{gcd}(n_1,\ldots,n_k)$ and $n_i'= \frac{n_i}{d}$; observe that $r(n_1,\ldots,n_k)= r(n_1',\ldots,n_k')$ and
$$r(n_1,\ldots,n_k)= r(m_1,\ldots,m_k) \Leftrightarrow (n_1',\ldots,n_k')= (m_1',\ldots,m_k') $$

From now until the end of the proof of Lemma \ref{lem:incompatible}, we fix a $K$-partition $\mcA$, with columns $C_1,\ldots,C_k$. For any $K$-partition $\mcB$ refining $\mcA$, and any column $D$ of $\mcB$, we let $n_i(D)$ denote the number of copies of $C_i$ contained in $D$. 

We define $r(D)=r(n_1(D),\ldots,n_k(D))$ and call it the \emph{repartition} of $D$; as above set $n_i'(D)=  \frac{n_i(D)}{\textrm{gcd}(n_1(D),\ldots,n_k(D))}$

\begin{lemma}\label{l:one_column}
Let $\mcB$ be a $K$-partition refining $\mcA$; assume that for all columns $D_1,D_2$ of $\mcB$ one has $r(D_1)=r(D_2)$. Then there exists a $K$-partition $\mcC$ refining $\mcA$ and which has a single column.
\end{lemma}

\begin{proof}
By assumption, $n_k'(D)$ does not depend on the column $D$ of $\mcB$, so we let $n_k'=n_k'(D)$ for some (any) column $D$ of $\mcB$.

For any column $D$ of $\mcB$, there exists some integer $M$ such that $D$ is made up of $M n_1'$ copies of $C_1$, \ldots, $M n_k'$ copies of $C_k$ stacked on top of each other. Up to reordering, we see that there exists a $K$-partition refining $\mcA$ such that each column is obtained by stacking $n_1'$ copies of $C_1$ on top of each other, then $n_2'$ copies of $C_2$,\ldots, then $n_k'$ copies of $C_k$; and repeating this pattern some number $M$ of times (where $M$ depends on the column). 

By separating at the beginning of each of these patterns, and putting the obtained subcolumns next to each other, we see that $\mcA$ has a refinement where all columns are made up of exactly $n_1'$ copies of $C_1$, \ldots, $n_k'$ copies of $C_k$ stacked on top of each other. We denote by $N$ the common height of these columns. We have also ensured that the $n_1'$ copies of $C_1$ are at the bottom, followed by $n_2'$ copies of $C_2$, etc.

Denote this refinement $\mcB=(B_{i,j})_{0 \le i \le m, 0 \le j < N}$, and set for all $j \in \{0,\ldots,N-1\}$
\[C_j = \bigsqcup_{0 \le i \le m} B_{i,j} \]
Then $(C_0,\ldots,C_{N-1})$ is a $K$-partition which refines $\mcA$ and has a single column (which consists of $n_1'$ copies of $C_1$ followed by $n_2'$ copies of $C_2$, and so on).
\end{proof}

\begin{lemma}\label{lem:incompatible}
Assume that $\mcB$ is a $K$-partition refining $\mcA$, and that there exist two columns $D_1$, $D_2$ of $\mcB$ with $r(D_1) \ne r(D_2)$. Then there exist $K$-partitions $\mcC$, $\mcD$ refining $\mcB$ such that for any $f \in G_\mcA$ one has $f O_\mcC f^{-1} \cap O_\mcD = \emptyset$.
\end{lemma}

\begin{proof}
To simplify notations below, we note that one can further refine $\mcB$ in such a way that there exist columns with distinct repartitions but the same height: if $D_1$ has height $a$ and $D_2$ has height $b$, form a new partition with one column obtained by stacking $b$ copies of $D_1$ on top of each other, and another by stacking $a$ copies of $D_2$ on top of each other; this is possible as long as one chooses a small enough base for these columns. These new columns have the same repartitions as $D_1$ and $D_2$, and both have height equal to $ab$. So we assume below that $\textrm{height}(D_1)=\textrm{height}(D_2)=H$. 

We use the fact that $D_1$, $D_2$ have different repartitions to build incompatible $\mcC$, $\mcD$ refining $\mcB$ . The intuition is as follows: first choose $\mcC$ so that every column of $\mcC$ begins with many copies of $D_1$; then every $K$-partition $\mcE$ refining $\mcC$ will be such that every column contains many consecutive copies of $D_1$, and to move from an atom of $\mcE$ to an atom at the bottom of these many successive copies of $D_1$ one only needs to move by at most the largest height of any column of $\mcC$. So, if every column of a $K$-partition $\mcD$ begins with many copies of $D_2$, we expect $\mcC$ and $\mcD$ to be incompatible. To turn this intuition into a proof, we use a counting argument.

We now turn to the details of this argument.  Denote by $m_i$ the number of copies of $C_i$ occuring in $D_2$ and set $S=\sum_{i=1}^k m_i$.

Consider some $g \in O_\mcA$, $x \in X$ and some $N \in \N$ bigger than the height of any column in $\mcA$. Let $i_1< \ldots< i_p$ be the indices in $\{0,\ldots,N\}$ for which $g^i(x)$ belongs to the base of $\mcA$, and note that $x,\ldots,g^{i_1-1}(x)$ all belong to the same column of $\mcA$, as do $g^{i_j}(x),\ldots,g^{i_{j+1}-1}(x))$ for all $j \in \{1,\ldots,p-1\}$ and $g^{i_p}(x),\ldots,g^N(x)$. For $i \in \{1,\ldots,k\}$, we let $n_{i,N}(g;x)$ denote the number of times $C_i$ has occured in this sequence, and let 
\[s_N(g;x)=\frac{1}{\sum_{i=1}^k n_{i,N}(x)}(n_{1,N}(x),\ldots,n_{k,N}(x))\]

Fix an integer $s \ge 2$ such that $\frac{kH}{s-1} < \|r(D_1)-r(D_2)\|_\infty$. Using the fact that $K$ is a dynamical simplex, we may build a $K$-partition $\mcC$ refining $\mcB$ and such that all the columns of $\mcC$ begin with $s$ copies of $D_1$. Set $N =s H -1$, and fix an integer $M$ which is larger than the height of any column of $\mcC$ (hence also $M \ge N$).

Note that for any $g \in O_{\mcC}$ and any $x \in X$ there exists some $i \in \{0,\ldots,M-1\}$ such that $s_N(g;g^i(x))=r(D_1)$. Indeed, there exists some $i \in \{0,\ldots,M-1\}$ such that $g^i(x)$ belongs to the base of $\mcC$, and then $g^i(x),\ldots,g^{N+i}(x)$ are going through $s$ copies of $D_1$. 

Thus we have the following:
\begin{equation}\label{eq:incompatibility}
\forall f \in G_\mcA \ \forall h \in f O_\mcC f^{-1}\ \forall x \in X \ \exists i \in \{0,\ldots,M-1\} \quad s_N(h;h^i(x))=r(D_1)
\end{equation}
Indeed, this follows from the observation in the previous paragraph and the fact that $\mcC$ refines $\mcA$ and every $f \in G_\mcA$ maps every atom of $\mcA$ to itself. 

Let $\mcD$ be a $K$-partition refining $\mcB$ and such that all columns of $\mcD$ begin with at least $3M$ copies of $D_2$; denote by $U$ the base of $\mcD$. Consider $g \in O_\mcD$ and  $x \in g^{j}(U)$ for some $j \in \{0,\ldots,M-1\}$. 
There exists $j_1 \in \{0,\ldots, H-1\}$ such that $g^{j_1}(x)$ belongs to the base of $D_2$. Since the points $g^{j_1}(x),\ldots,g^{j_1-1+(s-1)H}(x)$ visit $(s-1)$ copies of $D_2$ and make up all but $H$ of the elements of $g(x),\ldots,g^N(x)$, we see that for all $i \in \{1,\ldots,k\}$ one has
\[(s-1) m_i \le n_{i,N}(g;x) \le (s-1)m_i + H \]
Denote $T=\sum_{i=1}^k n_{i,N}(g;x)$; by summing these inequalities (which is overkill, but does the job) we get
\[(s-1)S \le T \le (s-1)S +kH\]
Thus
\[\forall i \in \{1,\ldots,k\} \quad  \frac{(s-1)m_ i}{(s-1)S+kH} - \frac{m_i}{S} \le \frac{n_{i,N}(g;x)}{T} - \frac{m_i}{S} \le \frac{(s-1)m_i+H}{(s-1)S} - \frac{m_i}{S} \]
The right-hand side of this inequality is smaller than $\frac{H}{s-1}$. 
Also 
\[ 0 \le \frac{m_i}{S} - \frac{(s-1)m_i}{(s-1)S+kH} = \frac{kH m_i}{S((s-1)S+kH)} \le \frac{kH}{(s-1)}\]
This shows that $\|r(D_2)-s_N(g;x)\|_\infty \le \frac{kH}{s-1}$, thus $s_N(g;x) \ne r(D_1)$. 

It follows that $s_N(g;g^i(x)) \ne r(D_1)$  for all $i \in \{0,\ldots,M-1\}$ and all $x \in U$. 
Using \eqref{eq:incompatibility}, we conclude that $g \not \in fO_\mcC f^{-1}$ for any $f \in G_\mcA$; since $g$ was an arbitrary element of $O_\mcD$, we obtain as promised that $f O_\mcC f^{-1} \cap O_\mcD = \emptyset$ for any $f \in G_\mcA$.
\end{proof}

We briefly recall the definition of an odometer: fix a sequence $\bar k = (k_i)$ of integers $\ge 2$, and let 
\[Y_{\bar k} = \prod_{i=0}^{+\infty} \{0,\ldots,k_i-1\} \]
Then $Y_{\bar k}$ is a Cantor space; the corresponding odometer is the map $T_{\bar k} \colon Y \to Y$ defined by ``adding $1$ with right-carry". Formally, if $y \in Y$ is such that $y(i)< k_i - 1$ for some $i$, then one finds the smallest such $i$ and sets 
\[T_{\bar k}(y)=(\!\!\!\!\!\!\!\!\! \underbrace{0,\ldots,0}_{\textrm{ a string of } i \textrm{ zeroes}}\!\!\! \!\!\! \!,y(i)+1,y_{i+1},y_{i+2},\ldots ) \]
If there is no such $i$, then we set $T_{\bar k}(y)(i)=0$ for all $i$. 

We say that a map $g \colon X \to X$ is an odometer if there exists a sequence $\bar k$ and a homeomorphism $h \colon Y_{\bar k} \to X$ such that $g=h T_{\bar k} h^{-1}$. 

An alternative (and equivalent) description of odometers, up to isomorphism, is as follows: let $n_i$ be an increasing sequence of integers, with $n_i$ dividing $n_{i+1}$ for all $i$. Let $G_i$ denote the cyclic group $\Z/n_i \Z$, and $\pi_{i,j}$ the natural projection from $G_i$ to $G_j$ for $i \ge j$. The inverse limit $G$ of the family $(G_i,(\pi_{i,j}))$ is a compact group. Let $u=(1,1,\ldots) \in G$; the subgroup generated by $u$ is dense in $G$. Topologically, $G$ is a Cantor space, and the associated odometer is the map $x \mapsto x+u$. The unique invariant measure associated to this odometer is the Haar measure on $G$.

Note that for any odometer there is a natural sequence of Kakutani-Rokhlin partitions $(\mcA_n)$ which have exactly one column each, obtained by taking as base of $\mcA_n$ the set $\{y \in Y_{\bar k} \colon \forall i \le n \ y(i)=0\}$. It is easily seen that the existence of such a sequence actually characterizes odometers (\cite[Theorem~4.6]{Bezuglyi2006}, a fact we will use below). This property also implies that odometers are strictly ergodic and saturated.

\begin{theorem}\label{t:comeager_orbits}
There exists a comeager conjugacy class in $G_K^*$ if, and only if $K$ is a singleton $\{\mu\}$, and $\mu$ is the unique invariant probability measure for some odometer $g$. In that case, the conjugacy class of $g$ is comeager in $G_K^*$.
\end{theorem}

\begin{proof}
Assume that there exists a comeager conjugacy class in $G_K^*$. Since \eqref{item:WAP} of Proposition \ref{prop:criterion} is satisfied, any $K$-partition $\mcA$ admits a refinement $\mcB$ with a unique column. Indeed, for $\mcB$ to witness that this condition holds, Lemma \ref{lem:incompatible} shows that it is necessary that all columns of $\mcB$ have the same repartition, and we saw in Lemma \ref{l:one_column} that if $\mcA$ admits such a refinement then it admits one with a single column.

This allows us to build a $K$-saturated $g \in G_K$ with a sequence of Kakutani-Rokhlin partitions which each have exactly one column, following the same construction as in \cite{Ibarlucia2016}. Such a $g$ is an odometer, hence $K=\{\mu\}$ where $\mu$ is the unique Borel invariant probability measure for $g$.

Conversely, let $g$ be an odometer, and let $\mu$ be the unique $g$-invariant Borel probability measure. Since $g$ is $\{\mu\}$-saturated, we know from Proposition \ref{prop:saturated} that $g$ has a dense conjugacy class in $G_{\mu}^*$. Denote this conjugacy class by $\Omega(g)$.

Given a Kakutani-Rokhlin partition $\mcA$ for $g$ with one column, we claim that 
\[ \left\{ h g h^{-1} \colon h \in G_\mcA \right\} = \Omega(g) \cap O_ \mcA \]
Indeed, inclusion from left-to-right is immediate. To see the converse inclusion, pick $f \in \Omega(g) \cap O_{\mcA}$; then $\mcA$ is a Kakutani-Rokhlin partition for both $f$ and $g$. Using the fact that $f,g$ are conjugate odometers, there exist sequences $(\mcB_n)$, $(\mcC_n)$ of Kakutani-Rokhlin partitions for $f$, $g$ respectively, which each generate $\Clopen(X)$ and are such that for all $n$ $\mcB_n$ and $\mcC_n$ each have exactly one column for all $n$, and the height of $\mcB_n$ is equal to the height of $\mcC_n$. We may additionally ensure that the bases of $\mcB_n$, $\mcC_n$ shrink to the same point which belongs to $\base(\mcA)$. There exists $n_0$ such that for all $n \ge n_0$ every atom of $\mcA$ is a union of atoms of both $\mcB_n$ and $\mcC_n$; it follows that both $\mcB_n$ and $\mcC_n$ refine $\mcA$ as $K$-partitions. One can then build a sequence (starting at $n_0$) of refining partial isomorphisms $h_n \colon \mcB_n \to \mcC_n$ such that $h_n(\alpha)=\alpha$ for every atom of $\mcA$, and for every atom $\beta$ of $\mcC_n$ one has $h_n f h_n ^{-1}(\beta)= g(\beta)$. The union of the sequence $(h_n)_{n \ge n_0}$ defines an automorphism $h$ of $\Clopen(X)$ such that $hfh^{-1}(U)=g(U)$ for every clopen $U$, and $h(\alpha)=\alpha$ for every atom $\alpha$ of $\mcA$. Hence there exists $h \in G_\mcA$ such that $hf h^{-1}=g$.

We just proved that the map $G_K \to \Omega(g)$, $h \mapsto hgh^{-1}$ is open, and then Effros' theorem (see e.g.~\cite[Theorem~3.2.4]{Gao2009a}) yields that $\Omega(g)$ is comeager in $\overline{\Omega(g)}= G_{\{\mu\}}^*$.
\end{proof}

The following consequence of Theorem \ref{t:comeager_orbits} was pointed out by A. Yingst.
\begin{cor}[Yingst]
Let $K$ be a dynamical simplex. There does not exist a comeager conjugacy class in $G_K^*$ exactly when some measure in $K$ gives some clopen set an irrational measure.
\end{cor}

\begin{proof}
A measure $\mu$ is the unique invariant measure associated to an odometer iff $\mu$ is good and assigns a rational measure to every clopen set (see \cite{Akin2005}*{Theorem~2.16}). This proves one implication above, as well as the converse implication in the particular case when $K$ is a singleton.

Assume now that $K$ is not a singleton, and fix $A \in \Clopen(X)$, $\mu_1, \mu_2 \in K$ such that $\mu_1(A) \ne \mu_2(A)$. Since $K$ is connected and $A \mapsto \mu(A)$ is continuous, there exists $\mu \in K$ such that $\mu(A)$ is irrational. 
\end{proof}

\bibliography{mybiblio}

\end{document}